\documentclass[11pt, titlepage]{article}\usepackage[]{graphicx}\usepackage[]{color}
\usepackage{alltt}
\usepackage{moreverb}
\usepackage{amsmath}
\usepackage{amssymb}
\usepackage{amsthm}
\usepackage{lineno}
\usepackage{wrapfig}
\usepackage{setspace}
\usepackage{titling}
\pretitle{\begin{center}\LARGE}
\posttitle{\par\end{center}\vskip 0.5em}
\preauthor{\large \lineskip 0.5em}
\DeclareRobustCommand{\authorthing}{
\begin{center}
Jay M. Ver Hoef$^{a}$, Ephraim M. Hanks$^{b}$, Mevin B. Hooten$^{c}$ \\
\vspace{.5cm} 
\begin{tabular}{rl}
\multicolumn{2}{c} {\hrulefill} \\
\multicolumn{2}{l} {$^a$ Marine Mammal Laboratory, NOAA Alaska Fisheries Science Center,} \\
{} & \hspace{.5cm} 7600 Sand Point Way NE, Seattle, WA 98115, tel: (907) 456-1995 \\
\multicolumn{2}{c}{and} \\
\multicolumn{2}{l}{$^b$ Department of Statistics, The Pennsylvania State University} \\
\multicolumn{2}{c}{and} \\
\multicolumn{2}{l} {$^{c}$ U.S. Geological Survey, Colorado Cooperative Fish and Wildlife Research Unit,} \\
{} & \hspace{.5cm}  Department of Fish, Wildlife, and Conservation Biology, \\
{} & \hspace{.5cm}  and Department of Statistics, Colorado State University \\
\multicolumn{2}{c} {\hrulefill} \\
\end{tabular}
\end{center}}
\author{\authorthing}
\postauthor{}
\predate{\begin{center}\large}
\postdate{\par\end{center}}
\usepackage{natbib,enumitem}

\newcommand\BibTeX{{\rmfamily B\kern-.05em \textsc{i\kern-.025em b}\kern-.08em
T\kern-.1667em\lower.7ex\hbox{E}\kern-.125emX}}

\newcommand\bb{\mathbf{b}}
\newcommand\bv{\mathbf{v}}
\newcommand\bx{\mathbf{x}}
\newcommand\bz{\mathbf{z}}
\newcommand\by{\mathbf{y}}
\newcommand\bA{\mathbf{A}}
\newcommand\bB{\mathbf{B}}
\newcommand\bC{\mathbf{C}}
\newcommand\bD{\mathbf{D}}
\newcommand\bE{\mathbf{E}}
\newcommand\bF{\mathbf{F}}
\newcommand\bG{\mathbf{G}}
\newcommand\bH{\mathbf{H}}
\newcommand\bI{\mathbf{I}}

\newcommand\bL{\mathbf{L}}
\newcommand\bM{\mathbf{M}}
\newcommand\bP{\mathbf{P}}
\newcommand\bQ{\mathbf{Q}}

\newcommand\bS{\mathbf{S}}
\newcommand\bU{\mathbf{U}}
\newcommand\bV{\mathbf{V}}
\newcommand\bW{\mathbf{W}}
\newcommand\bZ{\mathbf{Z}}
\newcommand\cI{\mathcal{I}}
\newcommand\cL{\mathcal{L}}
\newcommand\bvarepsilon{\mbox{\boldmath $\varepsilon$}}
\newcommand\bmu{\mbox{\boldmath $\mu$}}
\newcommand\bnu{\mbox{\boldmath $\nu$}}
\newcommand\btheta{\mbox{\boldmath $\theta$}}
\newcommand\bLambda{\mathbf{\Lambda}}
\newcommand\bSigma{\mathbf{\Sigma}}
\newcommand\bOmega{\mathbf{\Omega}}
\newcommand\bGamma{\mathbf{\Gamma}}
\newcommand\bzero{\mathbf{0}}
\newcommand\bone{\mathbf{1}}
\newcommand\upp{^{T}}
\newcommand\upi{^{-1}}
\newcommand\mun{\textrm{un}}
\newcommand\mrs{\textrm{rs}}
\newcommand\msp{\textrm{sp}}
\newcommand\mcg{\textrm{cg}}

\newtheorem{theorem}{Theorem}

\newtheorem{proposition}{Proposition}
\newtheorem{corollary}{Corollary}
\newtheorem{remark}{Remark}

\makeatletter
\def\namedlabel#1#2{\begingroup
    #2%
    \def\@currentlabel{#2}%
    \phantomsection\label{#1}\endgroup
}
\makeatother
\makeatletter
\newcommand*\bigcdot{\mathpalette\bigcdot@{.5}}
\newcommand*\bigcdot@[2]{\mathbin{\vcenter{\hbox{\scalebox{#2}{$\m@th#1\bullet$}}}}}
\makeatother
\IfFileExists{upquote.sty}{\usepackage{upquote}}{}
\begin{document}


\begin{titlepage}
\title {On the Relationship between Conditional (CAR) and Simultaneous (SAR) Autoregressive Models}
\end{titlepage}

\maketitle


\begin{abstract}
We clarify relationships between conditional (CAR) and simultaneous (SAR) autoregressive models.  We review the literature on this topic and find that it is mostly incomplete.  Our main result is that a SAR model can be written as a unique CAR model, and while a CAR model can be written as a SAR model, it is not unique. In fact, we show how any multivariate Gaussian distribution on a finite set of points with a positive-definite covariance matrix can be written as either a CAR or a SAR model. We illustrate how to obtain any number of SAR covariance matrices from a single CAR covariance matrix by using Givens rotation matrices on a simulated example.  We also discuss sparseness in the original CAR construction, and for the resulting SAR weights matrix.  For a real example, we use crime data in 49 neighborhoods from Columbus, Ohio, and show that a geostatistical model optimizes the likelihood much better than typical first-order CAR models.  We then use the implied weights from the geostatistical model to estimate CAR model parameters that provides the best overall optimization. \\
\hrulefill 
\end{abstract}


\section{Introduction}

\citet[][p. 8]{Cres:stat:1993} divides statistical models for data collected at spatial locations into two broad classes: 1) geostatistical models with continuous spatial support, and 2) lattice models, also called areal models \citep{Bane:Carl:Gelf:hier:2004}, where data occur on a (possibly irregular) grid, or lattice, with a countable set of nodes or locations. The two most common lattice models are the conditional autoregressive (CAR) and simultaneous autoregressive (SAR) models, both notable for sparseness of their precision matrices.  These autoregressive models are ubiquitous in many fields, including disease mapping \citep[e.g.,][]{Clay:Kald:empi:1987,Laws:stat:2013}, agriculture \citep{Cull:Glee:spat:1991,Besa:Higd:Baye:1999}, econometrics \citep{Anse:spat:1988,LeSa:Pace:intr:2009}, ecology \citep{Lich:Simo:Shri:Fran:spat:2002,Kiss:Carl:spat:2008}, and image analysis \citep{Besa:stat:1986,Li:mark:2009}. CAR models form the basis for Gaussian Markov random fields \citep{Rue:Held:Gaus:2005} and the popular integrated nested Laplace approximation methods \citep[INLA,][]{Rue:Mart:Chop:appr:2009}, and SAR models are popular in geographic information systems (GIS) with the GeoDa software \citep{Anse:Syab:Kho:geod:2006}. Hence, both CAR and SAR models serve as the basis for countless scientific conclusions.  Because these are the two most common classes of models for lattice data, it is natural to compare and contrast them.  There has been sporadic interest in studying the relationships between CAR and SAR models \citep[e.g.,][]{Wall:clos:2004}, and how one model might or might not be expressed in terms of the other \citep{Hain:spat:1990,Cres:stat:1993,Mart:some:1987,Wall:Gotw:appl:2004}, but there is little clarity in the existing literature on the relationships between these two classes of autoregressive models.  

Our goal is to clarify, and add to, the existing literature on the relationships between CAR and SAR covariance matrices, by showing that any positive-definite covariance matrix for a multivariate Gaussian distribution on a finite set of points can be written as either a CAR or a SAR covariance matrix, and hence any valid SAR covariance matrix can be expressed as a valid CAR covariance matrix, and vice versa.  This result shows that on a finite dimensional space, both SAR and CAR models are completely general models for spatial covariance, able to capture any positive-definite covariance.  While CAR and SAR models are among the most commonly-used spatial statistical models, this correspondence between them, and the generality of both models, has not been fully described before now.  These results also shed light on some previous literature.

This paper is organized as follows: In Section~\ref{sec:rev}, we review SAR and CAR models and lay out necessary conditions for these models. In Section~\ref{sec:relcarsar}, we provide theorems that show how to obtain SAR and CAR covariance matrices from any positive definite covariance matrix, which also establishes the relationship between CAR and SAR covariance matrices.  In Section~\ref{sec:ex}, we provide examples of obtaining SAR covariance matrices from a CAR covariance matrix on fabricated data, and a real example for obtaining a CAR covariance matrix for a geostatistical covariance matrix.  Finally, in Section~\ref{sec:discon}, we conclude with a detailed discussion of the incomplete results of previous literature.


\section{Review of SAR and CAR models} \label{sec:rev}

In what follows, we denote matrices with bold capital letters, and their $i$th row and $j$th column with small case letters with subscripts $i,j$; for example, the $i,j$th element of $\bC$ is $c_{i,j}$.  Vectors are denoted as lower case bold letters. Let $\bZ \equiv (Z_1,Z_2,\ldots,Z_n)\upp$ be a vector of $n$ random variables at the nodes of a graph (or junctions of a lattice). The edges in the graph, or connections in the lattice, define neighbors, which are used to model spatial dependency.


\subsection{SAR Models} \label{sec:SAR}

Consider the SAR model with mean zero. An explicit autocorrelation structure is imposed,
\begin{equation} \label{eq:sareta}
  \bZ = \bB\bZ + \bnu,
\end{equation}
where the $n \times n$ spatial dependence matrix, $\bB$, is relating $\bZ$ to itself, and $\bnu \sim \textrm{N}(\bzero,\bOmega)$, where $\bOmega$ is diagonal with positive values. These models are generally attributed to \citet{Whit:stat:1954}. Solving for $\bZ$, note that sites cannot depend on themselves so $\bB$ will have zeros on the diagonal, and that $(\bI - \bB)\upi$ must exist \citep{Cres:stat:1993, Wall:Gotw:appl:2004}, where $\bI$ is the identity matrix. Then $\bZ \sim \textrm{N}(\bzero,\bSigma_{\textrm{SAR}})$, where 
\begin{equation} \label{eq:sarcov}
    \bSigma_{\textrm{SAR}} = (\bI - \bB)\upi\bOmega(\bI - \bB\upp)\upi;
\end{equation}
see, for example, \citet[p. 409]{Cres:stat:1993}. The spatial dependence in the SAR model is due to the matrix $\bB$ which causes the simultaneous autoregression of each random variable on its neighbors. Note that $\bB$ does not have to be symmetric because it does not appear directly in the inverse of the covariance matrix (i.e., precision matrix).  The covariance matrix must be positive definite. For SAR models, it is enough that $(\bI - \bB)$ is nonsingular (i.e., that $(\bI - \bB)\upi$ exists), because the quadratic form, writing it as $(\bI - \bB)\upi\bOmega[(\bI - \bB)\upi]\upp$, with $\bOmega$ containing positive diagonal values, ensures $\bSigma_{\textrm{SAR}}$ will be positive definite. 

In summary, the following conditions must be met for $\bSigma_{\textrm{SAR}}$ in (\ref{eq:sarcov}) to be a valid SAR covariance matrix:
\begin{description}[labelindent=1cm, labelwidth=1cm, nosep]
	\begin{singlespace}
        \item[\namedlabel{S1}{S1}]  $(\bI-\bB)$ is nonsingular, 
        \item[\namedlabel{S2}{S2}]  $\bOmega$ is diagonal with positive elements, and 
        \item[\namedlabel{S3}{S3}]  $b_{i,i} = 0, \; \forall \; i$.
	\end{singlespace}
\end{description}


\subsection{CAR models} \label{sec:CAR}

The term ``conditional,'' in the CAR model, is used because each element of the random process is specified conditionally on the values of the neighboring nodes. Let $Z_i$ be a random variable at the $i$th location, again assuming that the expectation of $Z_i$ is zero for simplicity, and let $z_j$ be its realized value. The CAR model is typically specified as
\begin{equation} \label{eq:car2}
				Z_i|\bz_{-i} \sim \textrm{N}\left(\sum_{\forall c_{i,j}\neq 0} c_{i,j}z_j,m_{i,i}\right),
  \end{equation}
where $\bz_{-i}$ is the vector of all $z_j$ where $j \ne i$, $\bC$ is the spatial dependence matrix with $c_{i,j}$ as its $i,j$th element, $c_{i,i} = 0$, and $\bM$ is a diagonal matrix with positive diagonal elements $m_{i,i}$. Note that $m_{i,i}$ may depend on the values in the $i$th row of $\bC$. In this parameterization, the conditional mean of each $Z_i$ is weighted by values at neighboring nodes. The variance component, $m_{i,i}$, often varies with node $i$, and thus $\bM$ is generally nonstationary.  In contrast to SAR models, it is not obvious that (\ref{eq:car2}) leads to a full joint distribution for $\bZ$. \citet{Besa:spat:1974} used Brook's lemma \citep{Broo:dist:1964} and the Hammersley-Clifford theorem \citep{Hamm:Clif:Mark:1971,Clif:Mark:1990} to show that, when $(\bI-\bC)\upi\bM$ is positive definite, $\bZ \sim \textrm{N}(\bzero,\bSigma_{\textrm{CAR}})$, with
\begin{equation} \label{eq:carcov}
        \bSigma_{\textrm{CAR}} = (\bI-\bC)\upi\bM.
\end{equation}
$\bSigma_{\textrm{CAR}}$ must be symmetric, requiring
\begin{equation} \label{eq:CarSymmetry}
				\frac{c_{i,j}}{m_{i,i}}=\frac{c_{j,i}}{m_{j,j}}, \; \; \forall \; i,j.
\end{equation}
Most authors describe CAR models as the construction (\ref{eq:car2}), with condition that $\bSigma_\textrm{CAR}$ must be positive definite given the symmetry condition (\ref{eq:CarSymmetry}).  However, a more specific statement is possible on the necessary conditions for $(\bI - \bC)$, making a comparable condition to \ref{S1} for SAR models. We provide a novel proof, Proposition~\ref{pro:C1proof} in the Appendix, showing that if $\bM$ is positive definite along with (\ref{eq:CarSymmetry}) (forcing symmetry on $\bSigma_\textrm{CAR}$), it is only necessary for $(\bI-\bC)$ to have positive eigenvalues for $\bSigma_\textrm{CAR}$ to be positive definite.

In summary, the following conditions must be met for $\bSigma_{\textrm{CAR}}$ in (\ref{eq:carcov}) to be a valid CAR covariance matrix:
\begin{description}[labelindent=1cm, labelwidth=1cm, nosep]
	\begin{singlespace}
        \item[\namedlabel{C1}{C1}] $(\bI-\bC)$ has positive eigenvalues,
        \item[\namedlabel{C2}{C2}] $\bM$ is diagonal with positive elements,
        \item[\namedlabel{C3}{C3}]  $c_{i,i} = 0, \, \forall \, i$, and
        \item[\namedlabel{C4}{C4}] $c_{i,j}/m_{i,i}=c_{j,i}/m_{j,j}, \; \forall \; i,j$. 
	\end{singlespace}
\end{description}


\subsection{Weights Matrices} \label{sec:weightsMat}

In practice, $\bB = \rho_s\bW$ and $\bC = \rho_c\bW$ are usually used to construct valid SAR and CAR models, where $\bW$ is a weights matrix with $w_{i,j} \ne 0$ when locations $i$ and $j$ are neighbors, otherwise $w_{i,j} = 0$. Neighbors are typically pre-specified by the modeler. When $i$ and $j$ are neighbors, we often set $w_{i,j} = 1$, or use row-standardization so that $\sum_{j=1}^n w_{i,j} = 1$; that is, dividing each row in unstandardized $\bW$ by $w_{i,+} \equiv \sum_{j=1}^n  w_{i,j}$ yields an asymmetric row-standardized matrix that we denote as $\bW_+$. For CAR models, define $\bM_+$ as the diagonal matrix with $m_{i,i} = 1/w_{i,+}$, then (\ref{eq:CarSymmetry}) is satisfied.  The row-standardized CAR model can be written equivalently as
\begin{equation}\label{eq:bWone}
				\bSigma_+ = \sigma^2(\bI - \rho_c\bW_+)\upi\bM_+ = \sigma^2(\textrm{diag}(\bW\bone) - \rho_c\bW)\upi,
\end{equation}
where $\bone$ is a vector of all ones, $\sigma^2$ is an overall variance parameter, and diag($\cdot$) creates a diagonal matrix from a vector.  A special case of the CAR model, called the intrinsic autoregressive model (IAR) \citep{Besa:Koop:cond:1995}, occurs when $\rho_c = 1$, but the covariance matrix does not exist, so we do not consider it further.

There can be confusion on how $\rho$ is constrained for SAR and CAR models, which we now clarify. Suppose that $\bW$ has all real eigenvalues. Let $\{\lambda_i\}$ be the set of eigenvalues of $\bW$, and let $\{\omega_i\}$ be the set of eigenvalues of $(\bI - \rho\bW)$. Then, in the Appendix (Proposition 4), we show that $\omega_i = (1 - \rho\lambda_i)$. First, notice that if $\lambda_i = 0$, then $\omega_i = 1$ for all $\rho$. Hence,  $(\bI - \rho\bW)$ will be nonsingular for all $\rho \notin \{\lambda_i\upi\}$ whenever $\lambda_i\ne 0$, which is sufficient for SAR model condition~\ref{S1}.  Note that it is possible for all $\omega_i$ to be positive, even when $\rho\bW$ has some zero eigenvalues ($\lambda_i = 0$), and thus our result is more general than that of \citet{Li:Cald:Cres:beyo:2007}, who only consider the case when all $\lambda_i \ne 0$.  If any $\lambda_i \ne 0$, then at least two $\lambda_i$ are nonzero because $\textrm{tr}(\bW) = \sum_{i=1}^n\lambda_i = 0$.  If at least two eigenvalues are nonzero, then $\lambda_{[1]}$, the smallest eigenvalue of $\bW$, must be less than zero, and $\lambda_{[N]}$, the largest eigenvalue of $\bW$, must be greater than zero. Then $1/\lambda_{[1]} < \rho < 1/\lambda_{[N]}$ ensures that $(\bI - \rho\bW)$ has positive eigenvalues (Appendix, Proposition~\ref{pro:PDrhobounds}) and satisfies condition~\ref{C1} for CAR models.  For SAR models, if $(\bI - \rho\bW)$ has positive eigenvalues it is also nonsingular, so $1/\lambda_{[1]} < \rho < 1/\lambda_{[N]}$ provides a sufficient (but not necessary) condition for condition~\ref{S1}.

In practice, the restriction $1/\lambda_{[1]} < \rho < 1/\lambda_{[N]}$ is often used for both CAR and SAR models.  When considering $\bW_+$, the restriction becomes $1/\lambda_{[1]} < \rho < 1$, where usually $1/\lambda_{[1]} < -1$. \citet{Wall:clos:2004} shows irregularities for negative $\rho$ values near the lower bound for both SAR and CAR models, thus many modelers simply use $-1 < \rho < 1$.  In fact, in many cases, only positive autocorrelation is expected, so a further restriction is used where $0 < \rho < 1$ \citep[e.g.,][]{Li:Cald:Cres:beyo:2007}. For these constructions, $\rho$ typically has more positive marginal autocorrelation with increasing positive $\rho$ values, and more negative marginal autocorrelation with decreasing negative $\rho$ values \citep{Wall:clos:2004}.  There has been little research on the behavior of $\rho$ outside of these limits for SAR models.

Our goal is to develop relationships that allow a CAR covariance matrix, satisfying conditions \ref{C1} - \ref{C4}, to be obtained from a SAR covariance matrix, satisfying conditions \ref{S1} - \ref{S3}, and vice versa. We develop these in the next section, and, in the Discussion and Conclusions section, we contrast our results to the incomplete results of previous literature.


\section{Relationships between CAR and SAR models} \label{sec:relcarsar}

Assume a covariance matrix for a SAR model as given in (\ref{eq:sarcov}), and a covariance matrix for a CAR model as given in (\ref{eq:carcov}).  We show that any zero-mean Gaussian distribution on a finite set of points, $\bZ \sim \textrm{N}(\bzero,\bSigma)$, can be written with a covariance matrix parameterized either as a CAR model, $\bSigma=(\mathbf{I}-\mathbf{C})^{-1}\bM$, or as a SAR model, $\bSigma=(\mathbf{I}-\mathbf{B})^{-1}\bOmega (\mathbf{I}-\mathbf{B}^T)^{-1}$.  It is straightforward to generalize to the case where the mean is nonzero so, for simplicity of notation, we use the zero mean case.  A corollary is that any CAR covariance matrix can be written as a SAR covariance matrix, and vice versa.  Before proving the theorems, some preliminary results are useful.

\begin{proposition} \label{pro:zerodiagmult}
\begin{singlespace}
 If $\bD$ is a square diagonal matrix, and $\bQ$ is a square matrix with zeros on the diagonal, then, provided the matrices are conformable, both $\bD\bQ$ and $\bQ\bD$ have zeros on the diagonal.
\end{singlespace}
\end{proposition}
\begin{proof}  We omit the proof because it is apparent from the algebra of matrix products. \qedhere
\end{proof}
\begin{proposition} \label{pro:mult3inverses}
\begin{singlespace}
 Let $\bA$, $\bB$, and $\bC$ be square matrices. If $\bA = \bB\bC$, and $\bA$ and $\bC$ have inverses, then $\bB$ has a unique inverse. 
\end{singlespace}
\end{proposition}
\begin{proof}
\begin{singlespace}
	Because $\bC$ has an inverse, $\bB = \bA\bC\upi$, and because $\bA$ has an inverse, $\bB\upi = \bC\bA\upi$.  $\bB\upi$ is unique because it is square and full-rank \citep[e.g.,][p. 80]{Harv:matr:1997}. \qedhere
\end{singlespace}
\end{proof}

We now prove that both SAR and CAR covariance matrices are sufficiently general to represent any finite-dimensional positive-definite covariance matrix.
\begin{theorem} \label{SARtheorem}
\begin{singlespace}
Any positive definite covariance matrix $\bSigma$ can be expressed as the covariance matrix of a SAR model $(\mathbf{I}-\mathbf{B})^{-1}\bOmega (\mathbf{I}-\mathbf{B}^T)^{-1}$, (\ref{eq:sarcov}), for a (non-unique) pair of matrices $\bB$ and $\bOmega$.
\end{singlespace}
\end{theorem}
\begin{proof}
\begin{singlespace}
We consider a constructive proof and show that the matrices $\bB$ and $\bOmega$ satisfy conditions \ref{S1} - \ref{S3}.  
\begin{enumerate}
\item [(i)] Write $\bSigma\upi = \bL\bL\upp$, and suppose that $\bL$ is full rank with positive eigenvalues. Note that $\bL$ is \emph{not} unique.  A Cholesky decomposition could be used, where $\bL$ is lower triangular, or a spectral (eigen) decomposition could be used, where $\bSigma = \bV\bE\bV\upp$, with $\bV$ containing orthonormal eigenvectors and $\bE$ containing eigenvalues on the diagonal and zeros elsewhere. Then $\bL = \bV\bE^{-1/2}$, where the diagonal matrix $\bE^{-1/2}$ contains reciprocals of square roots of the eigenvalues in $\bE$.  
\item [(ii)] Decompose $\bL$ into $\bL=\bG - \bP$ where $\bG$ is diagonal and $\bP$ has zeros on the diagonal.  Then $\bL\bL\upp = (\bG - \bP)(\bG\upp - \bP\upp)$ by construction.
\item [(iii)] Then set
\begin{equation}\label{eq:SARproof}
 \bOmega^{-1} = \mathbf{GG} \  \text{ and }\ \bB^T = \mathbf{PG}^{-1}.
\end{equation}
Note that because $\mathbf{L}$ has positive eigenvalues, then $\ell_{i,i}>0$, and because $\mathbf{G}$ is diagonal with $g_{i,i}=\ell_{i,i}$, $\mathbf{G}^{-1}$ exists.
\end{enumerate}
Then $\bSigma^{-1}=(\bI-\bB^T)\bOmega^{-1}(\bI-\bB)$, expressed in SAR form (\ref{eq:sarcov}).  The matrices $\bB$ and $\bOmega$ satisfy \ref{S1} - \ref{S3}, as follows.
\begin{enumerate}
\item [(\ref{S1})] Note that $\bP=\bB\upp\bG$, so $\bL=\bG-\bP=(\bI-\bB\upp)\bG$ and $\bL\upp = \bG(\bI - \bB)$.  Then, by Proposition \ref{pro:mult3inverses}, $(\bI-\bB)^{-1}$ and $(\bI-\bB\upp)^{-1}$ exist.
\item [(\ref{S2})] Because $\bG$ is diagonal, $\bOmega$ is diagonal with $\omega_{i,i}=g_{i,i}^2>0$.
\item [(\ref{S3})] By Proposition \ref{pro:zerodiagmult}, $b_{i,i}=0$ because $\bB^T=\bP\bG^{-1}$. \qedhere
 \end{enumerate}
\end{singlespace}
\end{proof}  


\begin{theorem} \label{CARtheorem}
\begin{singlespace}
Any positive-definite covariance matrix $\bSigma$ can be expressed as the covariance matrix of a CAR model $(\mathbf{I}-\mathbf{C})^{-1}\bM$, (\ref{eq:carcov}), for a unique pair of matrices $\bC$ and $\bM$ \citep[p. 434]{Cres:stat:1993}.
\end{singlespace}
\end{theorem}
\begin{proof}
\begin{singlespace}
We add an explicit, constructive proof of the result given by \citet[p. 434]{Cres:stat:1993} by showing that matrices $\bC$ and $\bM$ are unique and satisfy conditions \ref{C1} - \ref{C4}.
\begin{enumerate}
\item [(i)] Let $\mathbf{Q}=\bSigma^{-1}$ and decompose it into $\mathbf{Q}=\mathbf{D}-\mathbf{R}$, where $\mathbf{D}$ is diagonal with elements $d_{i,i}=q_{i,i}$ (the diagonal elements of the precision matrix $\mathbf{Q}$), and $\mathbf{R}$ has zeros on the diagonal ($r_{i,i}=0$) and off-diagonals equal to $r_{i,j}=-q_{i,j}$.
\item [(ii)] Set
  \begin{equation}\label{eq:CARproof}
\mathbf{C} = \mathbf{D}^{-1}\mathbf{R} \ \text{ and } \ \mathbf{M} = \mathbf{D}^{-1}.
\end{equation}
\end{enumerate}
Then $\bSigma^{-1}= \mathbf{D}-\mathbf{R} = \mathbf{D}(\mathbf{I}-\mathbf{D}^{-1}\mathbf{R}) = \bM^{-1}(\mathbf{I}-\mathbf{C})$, with $\bSigma$ expressed in CAR form (\ref{eq:carcov}).  The matrices $\mathbf{C}$ and $\mathbf{M}$, from (\ref{eq:CARproof}), are uniquely determined by $\bSigma$ because $\bSigma$ and $\bD$ have unique inverses, and satisfy \ref{C1} - \ref{C4}, as follows.
\begin{enumerate}
\item [(\ref{C1})] $\bM$ is strictly diagonal with positive values, so $\bM$ and $\bM\upi$ are positive definite.  By hypothesis, $\bSigma$, and hence $\bSigma\upi$ are positive definite.  Then $\bSigma\upi\bM = (\bI - \bC)$, so by Proposition~\ref{pro:C1proof} in the Appendix, $(\bI - \bC)$ has positive eigenvalues.
\item [(\ref{C2})] $m_{i,i}=1/q_{i,i}$, and because $\mathbf{Q}=\bSigma^{-1}$ is positive definite, we have that $q_{i,i}>0,i=1,2,\ldots,n$.  Thus, each $m_{i,i}>0$.  By construction, $m_{i,j}=0$ for $i\neq j$.  
\item [(\ref{C3})] By Proposition 2.1, $c_{i,i}=0$ because $\mathbf{C}=\mathbf{D}^{-1}\mathbf{R}$.
\item [(\ref{C4})] For $i\neq j$, we have that $c_{i,j}=d_{i,i}^{-1}r_{i,j}$.  As $m_{i,i}=d_{i,i}^{-1}=q_{i,i}$, we have that
\[\frac{c_{i,j}}{m_{i,i}} = \frac{d_{i,i}^{-1}r_{i,j}}{d_{i,i}^{-1}}=r_{i,j}=-q_{i,j}.\]
Because $\mathbf{Q}=\bSigma^{-1}$ is symmetric, $q_{i,j}=q_{j,i}$ and $c_{i,j}/m_{i,i} = c_{j,i}/m_{j,j}$. \qedhere
\end{enumerate}
\end{singlespace}
\end{proof}  


Having shown that any positive definite matrix $\bSigma$ can be expressed as either the covariance matrix of a CAR model or the covariance matrix of a SAR model, we have the following corollary.

\begin{corollary} \label{SAReqCAR}
\begin{singlespace}
Any SAR model can be written as a unique CAR model, and any CAR model can be written as a non-unique SAR model.
\end{singlespace}
\end{corollary}
\begin{proof}
\begin{singlespace}
The proof follows directly by first noting that a SAR model yields a positive-definite covariance matrix, and applying Theorem~\ref{CARtheorem}, and then noting that a CAR model yields a positive-definite covariance matrix, and applying Theorem~\ref{SARtheorem}. \qedhere
\end{singlespace}
\end{proof}

The following corollary gives more details on the non-unique nature of the SAR models.

\begin{corollary} \label{infSARcorol}
\begin{singlespace}
				Any positive-definite covariance matrix can be expressed as one of an infinite number of $\bB$ matrices that define the SAR covariance matrix in (\ref{eq:sarcov}).
\end{singlespace}
\end{corollary}
\begin{proof} 
\begin{singlespace}
Write $\bSigma\upi = \bL\bL\upp$ as in Theorem~\ref{SARtheorem}. Let $\bA_{h,s}(\theta)$ be a Givens rotation matrix \citep{Golu:Van:matr:2012}, which is a sparse orthonormal matrix that rotates angle $\theta$ through the plane spanned by the $h$ and $s$ axes.  The elements of $\bA_{h,s}(\theta)$ are as follows.  For $i \notin \{h,s\},\ a_{i,i}=1$.  For $i \in \{h,s\}, \ a_{h,h}=a_{s,s}=\cos(\theta),\ a_{h,s}=\sin(\theta)$ and $a_{s,h}=-\sin(\theta)$.  All other entries of $\bA_{h,s}(\theta)$ are equal to zero. Notice that $\bSigma\upi = \bL\bL\upp = \bL(\bA_{h,s}\upp(\theta)\bA_{h,s}(\theta))\bL\upp = \bL_*\bL_*\upp$, where $\bL_*=\bL\bA_{h,s}\upp(\theta)$.  A SAR covariance matrix can be developed as readily for $ \bL_*$ as for $\bL$ in the proof of Theorem~\ref{SARtheorem}.  Any of the infinite values of $\theta \in [0,2\pi)$ will result in a unique $\bA_{h,s}(\theta)$, leading to a different $\bL_*$, and a different $\bB$ matrix in (\ref{eq:SARproof}), but yielding the same positive-definite covariance matrix $\bSigma$. \qedhere
\end{singlespace}
\end{proof}

\begin{wrapfigure}{R}{0.5\textwidth}
  \centering
  \includegraphics[width=0.49\textwidth,keepaspectratio]{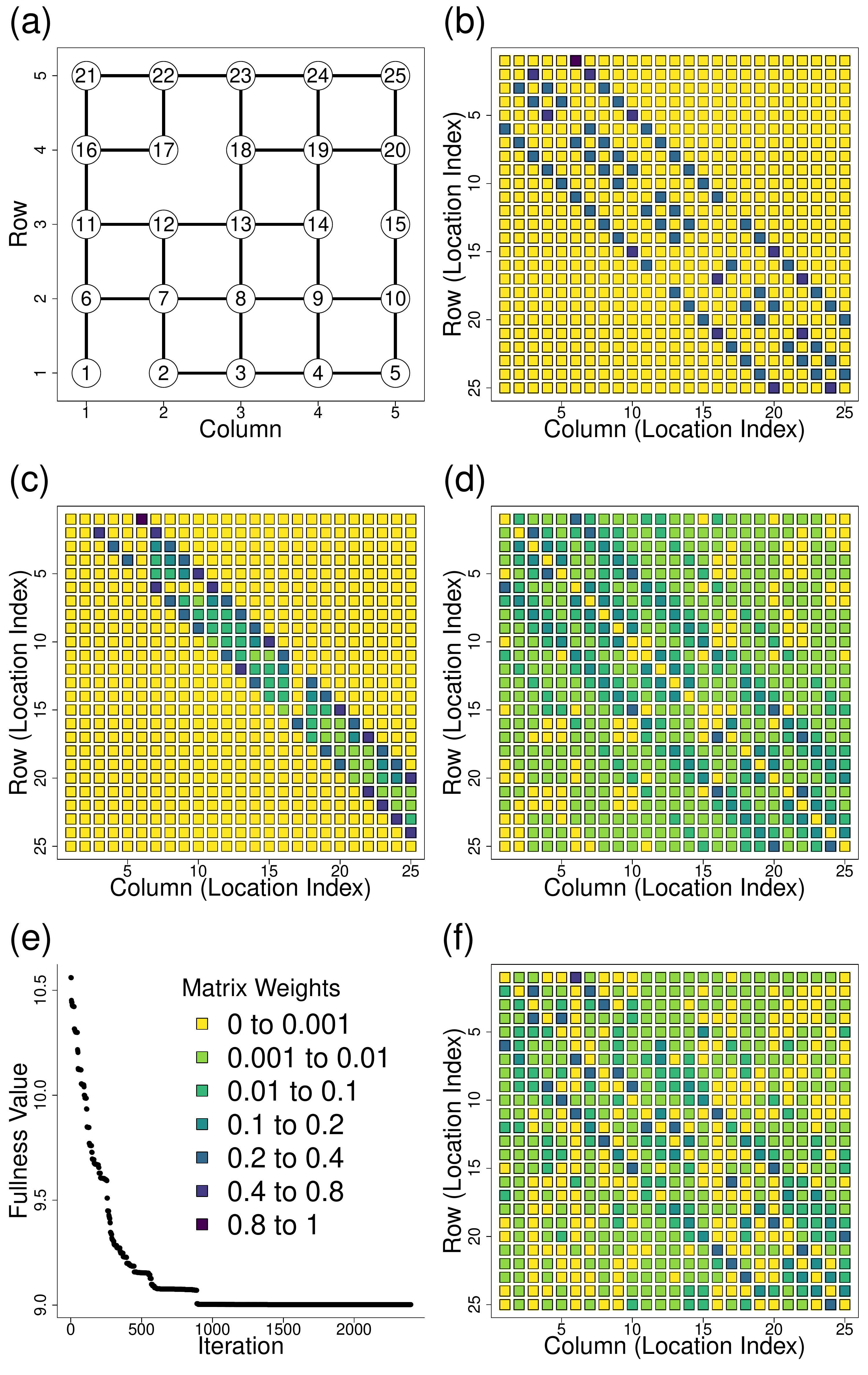}
  \caption{\small Sparseness in CAR and SAR models. (a) 5 $\times$ 5 grid of spatial locations, with lines connecting neighboring sites. The numbers in the circles are indexes of the locations. (b) Graphical representation of weights in the $\rho\bW_+$ matrix in the CAR model. The color legend is given below. (c) Graphical representation of weights in the $\bB$ matrix when using the Cholesky decomposition, and (d) when using spectral decomposition. (e) Fullness function during minimization when searching for sparseness. (f) Graphical representation of weights in the $\bB$ matrix at the termination of an algorithm to search for sparseness using Givens rotations on the spectral decomposition in (d).    \label{Fig-graphModel}}
\end{wrapfigure}


\subsection{Implications of Theorems and Corollaries} \label{sec:implications}

Note that for Corollary~\ref{infSARcorol}, additional $\bB$ matrices that define a fixed positive-definite covariance matrix in Corollary~\ref{infSARcorol} could also be obtained by repeated Givens rotations.  For example, let $\bL_*=\bL\bA_{1,2}\upp(\theta)\bA_{3,4}\upp(\eta)$ for angles $\theta$ and $\eta$.  Then a new $\bB$ can be developed for this $\bL_*$ just as readily as those in the proof to Corollary \ref{infSARcorol}.  We use this idea extensively in the examples.

Theorem~\ref{SARtheorem} helps clarify the use of $\bOmega$. Authors often write the SAR model as $(\bI - \bB)\upi(\bI - \bB\upp)\upi$, assuming that $\bOmega = \bI$ in (\ref{eq:sarcov}). In the proofs to Theorem~\ref{SARtheorem} and Corollary~\ref{infSARcorol}, this requires finding $\bL$ with ones on the diagonal so that $\bG = \bI$.  It is interesting to consider if one can always find such $\bL$, which would justify the practice of using the simpler form, $(\bI - \bB)\upi(\bI - \bB\upp)\upi$, for SAR models.  We leave that as an open question.

In Section~\ref{sec:weightsMat}, we discussed how most CAR and SAR models are constructed by constraining $\rho$ in $\rho\bW$. Consider Theorem \ref{SARtheorem}, where $\bL$ is a lower-triangular Cholesky decomposition. Then $\bP$ has zero diagonals and is strictly lower triangular, and so $\bB^T = \mathbf{PG}^{-1}$ is strictly lower triangular. In this construction, all of the eigenvalues of $\bB$ are zero.   Thus, for SAR models, there are unexplored classes of models that do not depend on the typical construction $\bB = \rho\bW$.  

Most CAR and SAR models are developed such that $\bC$ and $\bB$ are sparse matrices, containing mostly zeros, but containing positive elements whose weights depend locally on neighbors.  Although we demonstrated how to obtain a CAR covariance matrix from a SAR covariance matrix, and vice versa, there is no guarantee that using a sparse $\bC$ in a CAR model will yield a sparse $\bB$ in a SAR model, or vice versa.  We explore this idea further in the following examples.


\section{Examples} \label{sec:ex}

We provide two examples, one where we illustrate Theorem~\ref{SARtheorem} primarily, and a second where we use Theorem~\ref{CARtheorem}.   In the first, we fabricated a simple neighborhood structure and created a positive definite matrix by a CAR construction. Using Givens rotation matrices, we then obtained various non-unique SAR covariance matrices from the CAR covariance matrix. We also explore sparseness in $\bB$ for SAR models when they are obtained from sparse $\bC$ for CAR models.

For a second example, we used real data on neighborhood crimes in Columbus, Ohio. We model the data with the two most common CAR models, using a first-order neighborhood model where $\bC$ is both unstandardized and row-standardized. Then, from a positive-definite covariance matrix obtained from a geostatistical model, we obtain the equivalent and unique CAR covariance matrix.  We use the weights obtained from the geostatistical covariance matrix to allow further CAR modeling, finding a better likelihood optimization than both the unstandardized and row-standardized first-order CAR models.

Consider the graph in Figure~\ref{Fig-graphModel}a, which shows an example of neighbors for a CAR model. Using one to indicate a neighbor, and zeros elsewhere, the $\bW$ matrix was used to create the row-standardized $\bW_+$ matrix in (\ref{eq:bWone}). Values of $\rho_c\bW_+$, where $\rho_c=0.9$, are shown graphically in Figure~\ref{Fig-graphModel}b. For the resulting covariance matrix, $\bSigma_{+}$ in (\ref{eq:bWone}), the Cholesky decomposition was used to create $\bL$ as in Theorem~\ref{SARtheorem}. Using (\ref{eq:SARproof}) in Theorem~\ref{SARtheorem}, the weights matrix $\bB$ created from $\bL$ is shown in Figure~\ref{Fig-graphModel}c.  For the same covariance matrix $\bSigma_{+}$, we also used the spectral decomposition to create $\bL$ as in Theorem~\ref{SARtheorem}. The weights matrix $\bB$ created from this $\bL$, using (\ref{eq:SARproof}) in Theorem~\ref{SARtheorem}, is shown in Figure~\ref{Fig-graphModel}d.  Note that the $\bB$ matrix in Figure~\ref{Fig-graphModel}d is less sparse than $\bB$ in Figure~\ref{Fig-graphModel}c, although they both yield exactly the same covariance matrix by the SAR construction (\ref{eq:sarcov}), which we verified numerically. Figure~\ref{Fig-graphModel}c also verifies our comments in Section~\ref{sec:implications}; that there exists some $\bB$ where all eigenvalues are zero (because all diagonal elements are zero).

We also sought to transform the $\bB$ matrix in Figure~\ref{Fig-graphModel}d to a sparser form using the proof to Corollary~\ref{infSARcorol} and the Given's rotations. For a vector $\bx$ of length $n$, an index of sparseness \citep{Hoye:non:2004} is
\[
  \textrm{sparseness}(\bx) = \frac{\sqrt{n} - \frac{\sum_i|x_i|}{\sqrt{\sum_i x_i^2}}}{\sqrt{n}-1},
\]
which ranges from zero to one.  Ignoring the dimensions of a matrix, we create the matrix function
\[
  f(\bB) = \frac{\sum_{i,j}|b_{i,j}|}{\sqrt{\sum_{i,j} b_{i,j}^2}},
\]
which is a measure of the fullness of a matrix.  We propose an iterative algorithm to minimize $f(\bB)$ for orthonormal Givens rotations as explained in Corollary~\ref{infSARcorol}.  Let $\bL_{h,s}(\theta) = \bL\bA_{h,s}\upp(\theta)$, where $\bL = \bV\bE^{-1/2}\bV\upi$ used the spectral decomposition of $\bSigma_{+}$ as in the proof of Theorem~\ref{SARtheorem}, and $\bA_{h,s}(\theta)$ is a Givens rotation matrix as in the proof of Corollary~\ref{infSARcorol}. Denote $\theta^*_k$ as the value of $\theta$ that minimizes $f(\bB)$ when $\bB$ is created by decomposing $\bL\bA_{h,s}\upp(\theta)$ into $\bP$ and $\bG$ (as in $(ii)$ in Theorem~\ref{SARtheorem}), while constraining $\theta$ to values satisfying $b_{i,j} \ge 0 \ \forall \ i,j$. Then $\bL_{1,2}^{[1]} \equiv \bL\bA_{1,2}\upp(\theta^*_1)$, where $k=1$ is the first iteration.  For the second iteration, let $\theta^*_2$ be the value that minimizes $f(\bB)$ for $\bB$ created from $\bL_{1,2}^{[1]}\bA_{1,3}\upp(\theta)$, and hence for $k=2$, $\bL_{1,3}^{[2]} \equiv \bL_{1,2}^{[1]}\bA_{1,3}\upp(\theta^*_2)$. We cycled through $h = 1,2,\ldots,24$ and $s = (h+1),\ldots,25$ for each iteration $k$ in a coordinate decent minimization of $f(\bB)$. We cycled through all of $h$ and $s$ eight times for a total of $8(25)(25-1)/2 = 2400$ iterations.  The value of $f(\bB)$ for each iteration is plotted in Figure~\ref{Fig-graphModel}e and the final $\bB$ matrix is given in Figure~\ref{Fig-graphModel}f. Although we did not achieve the sparsity of Figure~\ref{Fig-graphModel}c, we were able to increase sparseness from the starting matrix in Figure~\ref{Fig-graphModel}d.  Note that the $\bB$ matrix depicted in Figure~\ref{Fig-graphModel}f yields exactly the same covariance matrix as the $\bB$ matrices shown in Figures~\ref{Fig-graphModel}c,d. There are undoubtedly better ways to minimize $f(\bB)$, such as simulated annealing \citep{Kirk:Gela:Vecc:opti:1983}, and there may be alternative optimization criteria.  We do not pursue these here. Our goal was to show that it is possible to explore many configurations of matrix weights in SAR models, which produce equivalent covariance matrices, by using orthonormal Givens rotations of the $\bL$ matrix. 

\newpage
\subsection{Columbus Crime Data} \label{sec:crime}




\begin{wrapfigure}{L}{0.5\textwidth}
  {\linespread{1.0}
  \centering
  \includegraphics[width=.49\textwidth]{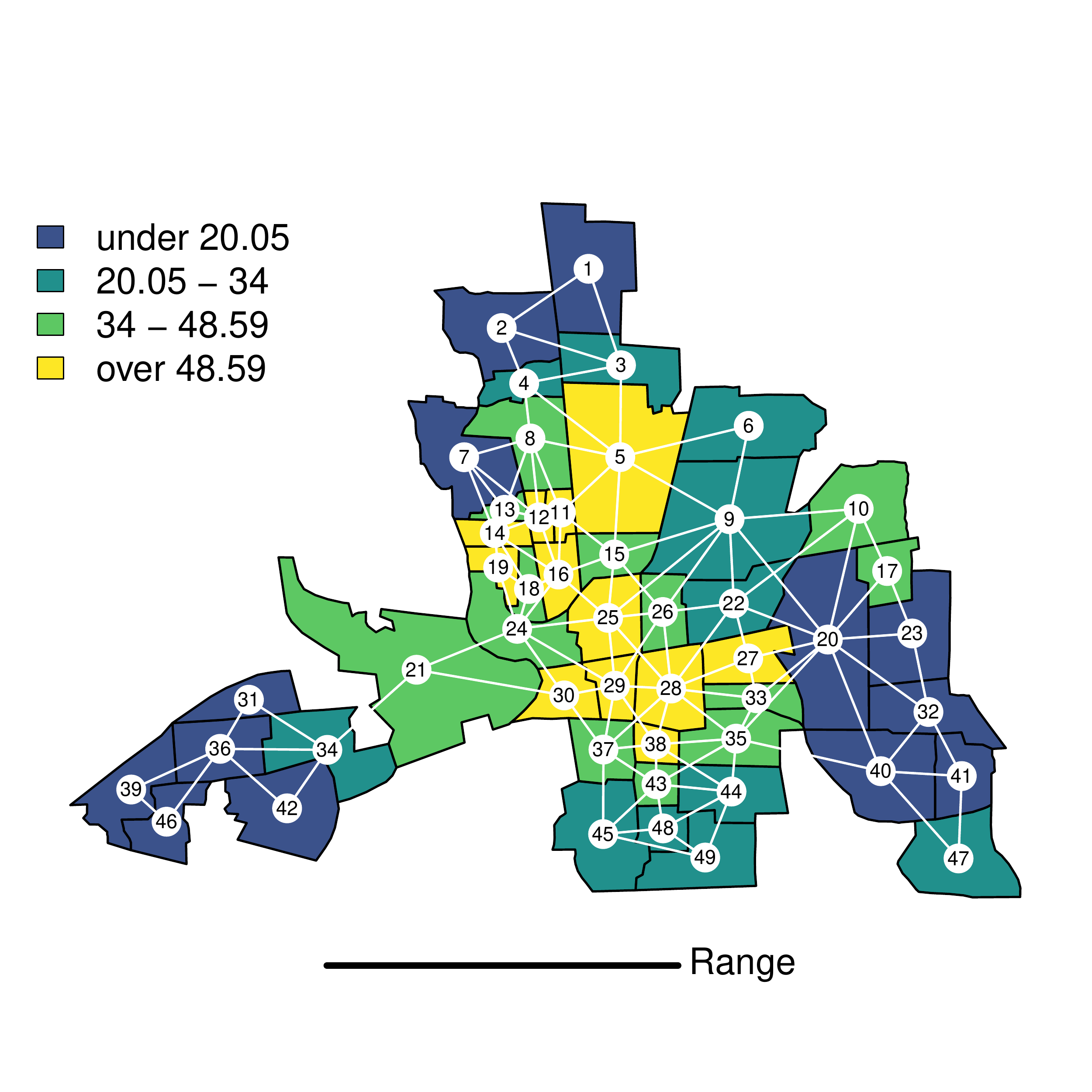}
  \caption{\small Columbus crime map, in rates per 1000 people.  Numbers in each polygon are the indexes for locations, and the white lines show first-order neighbors.  The estimated range parameter from the spherical geostatistical model is shown at the bottom.    \label{Fig-ColumbusCrimeMap}}}         
\end{wrapfigure}

The Columbus data are found in the \texttt{spdep} package \citep{Biva:Hauk:Koss:comp:2013,Biva:Pira:comp:2015} for \texttt{R} \citep{R:Deve:Core:ALan:2016}. Figure~\ref{Fig-ColumbusCrimeMap} shows 49 neighborhoods in Columbus, Ohio.  We used residential burglaries and vehicle thefts per thousand households in the neighborhood \citep[][Table 12.1, p. 189]{Anse:spat:1988} as the response variable. Spatial pattern among neighborhoods appeared autocorrelated (Figure~\ref{Fig-ColumbusCrimeMap}), with higher crime rates in the more central neighborhoods.  When analyzing rate data, it is customary to account for population size \citep[e.g.,][]{Clay:Kald:empi:1987}, which affects the variance of the rates. However, for illustrative purposes, we used raw rates.  A histogram of the data appeared approximately bell-shaped, thus we assumed a Gaussian distribution with a covariance matrix containing autocorrelation among locations.


\begin{wrapfigure}{R}{0.95\textwidth}
  {\linespread{1.0}
  \centering
  \includegraphics[width=.8\textwidth]{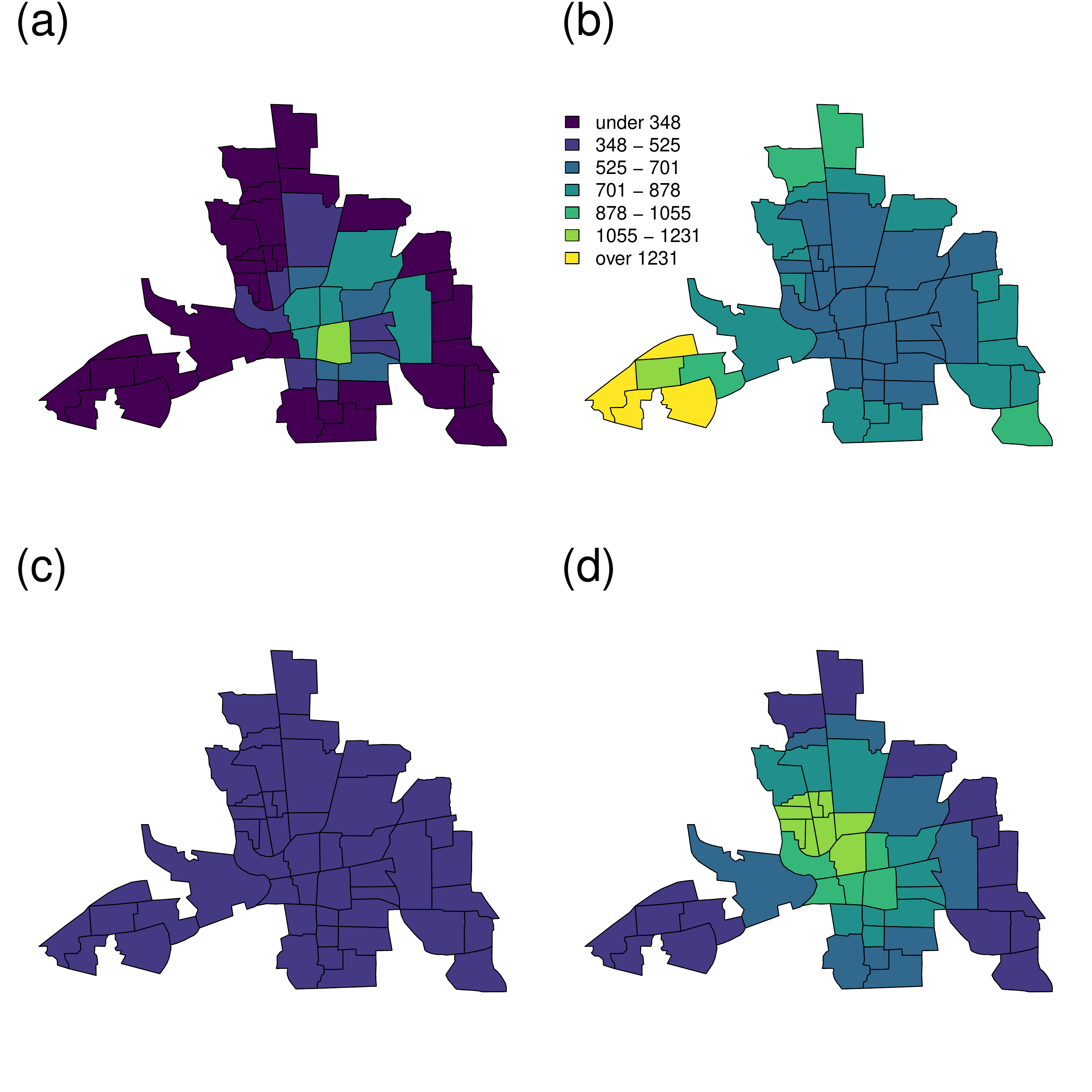}
  \caption{\small Marginal variances by location. (a) Unstandardized first-order CAR model, (b) Row-standardized first-order CAR model, (c) spherical geostatistical model, (d) CAR model using weights obtained from geostatistical model. \label{Fig-MargVar}}         
}         
\end{wrapfigure}

First-order neighbors were also taken from the \texttt{spdep} package for \texttt{R}, and are shown by white lines in Figure~\ref{Fig-ColumbusCrimeMap}.  Using a one to indicate a neighbor, and zero otherwise, we denote the $49 \times 49$ matrix of weights as $\bW_\mun$, and the CAR precision matrix has $\bC = \rho_\mun\bW_\mun$ and $\bM = \sigma_\mun^2\bI$ in (\ref{eq:carcov}). Using the eigenvalues of $\bW_\textrm{un}$, the bounds for $\rho_\mun$ were -0.335 $< \rho_\mun <$ 0.167. We added a constant independent diagonal component, $\delta^2_\mun\bI$ (also called the nugget effect in geostatistics), so the covariance matrix was $\bSigma_\mun = \sigma^2_\mun(\bI - \rho_\mun\bW_\mun)\upi + \delta_\mun^2\bI$.  Denote the crime rates as $\by$. We assumed a constant mean, so $\by \sim \textrm{N}(\bone\mu,\bSigma_\textrm{un})$, where $\bone$ is a vector of all ones. Let $\cL(\btheta_\mun|\by)$ be minus 2 times the restricted maximum likelihood equation \citep[REML,][]{Patt:Thom:reco:1971,Patt:Thom:maxi:1974} for the crime data, where the set of covariance parameters is $\btheta_\mun = (\sigma^2_\mun,\rho_\mun, \delta^2_\mun)\upp$. We optimized the likelihood using REML and obtained $\cL(\hat{\btheta}_\mun|\by)=$ 388.83.   Recall that CAR models have nonstationary variances and covariances \citep[e.g.,][]{Wall:clos:2004}.  The marginal variances of the estimated model are shown in Figure~\ref{Fig-MargVar}a, and the marginal correlations are shown in Figure~\ref{Fig-MargCorr}a.

\begin{wrapfigure}{R}{0.5\textwidth}
  {\linespread{1.0}
  \centering
  \includegraphics[width=.49\textwidth]{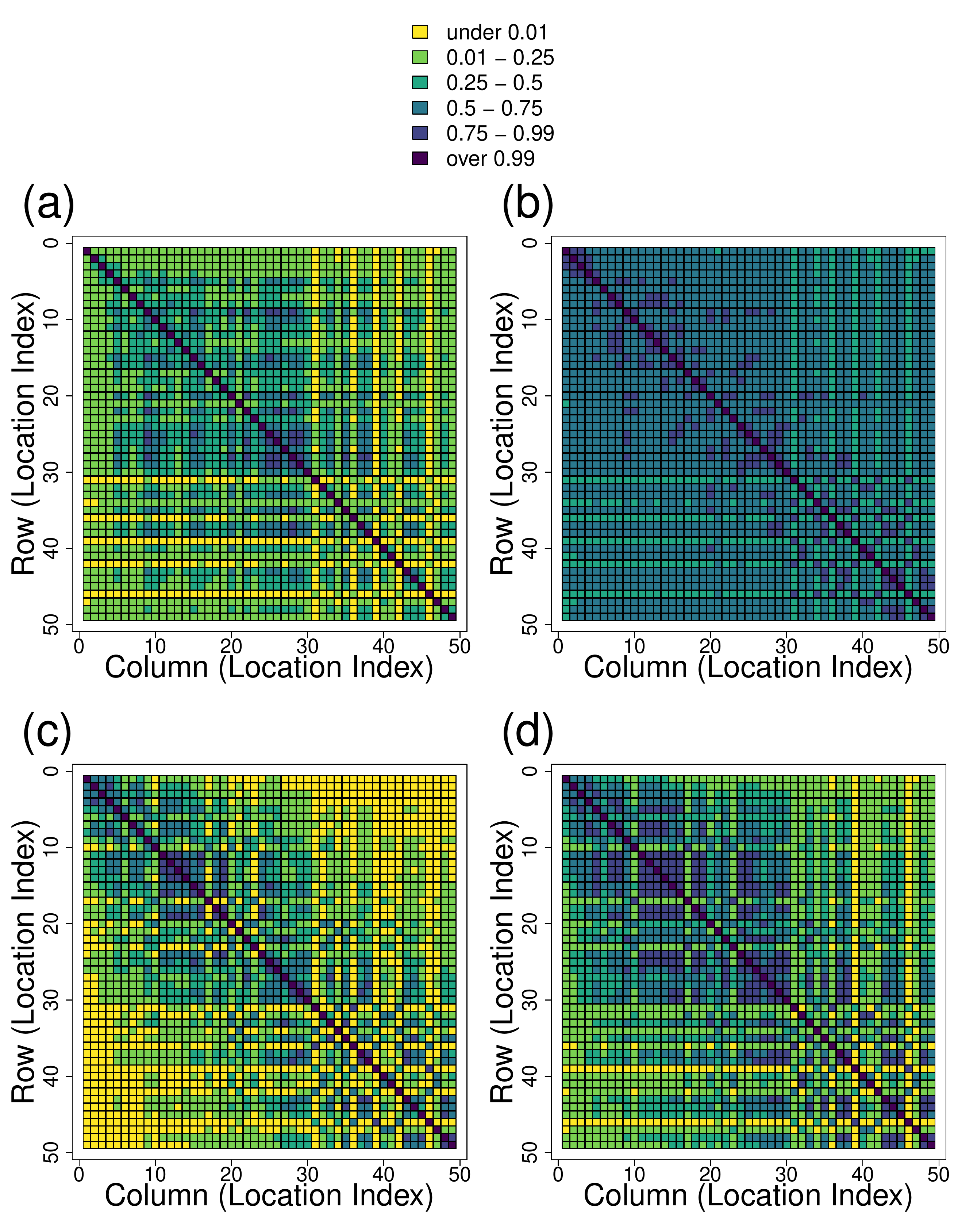}
  \caption{\small Marginal correlations, none of which were below zero. The location indexes are given by the numbers in Figure~\ref{Fig-ColumbusCrimeMap}. (a) Unstandardized first-order CAR model, (b) Row-standardized first-order CAR model, (c) spherical geostatistical model, (d) CAR model using weights obtained from geostatistical model. \label{Fig-MargCorr}}         
}         
\end{wrapfigure}

\vspace{.5cm}
We also optimized the likelihood using the row-standardized weights matrix, $\bW_+$ in (\ref{eq:bWone}), which we denote $\bW_\mrs$.  In this case, the CAR precision matrix has $\bC = \rho_\mrs\bW_+$, $-1 <\rho_\mrs < 1$, and $\bM = \sigma^2_\mrs\bM_+$ in (\ref{eq:carcov}).  Again we added a nugget effect, so $\bSigma_\mrs = \sigma^2_\mrs(\bI - \rho_\mun\bW_+)\upi\bM_+ + \delta_\mrs^2\bI$. For the set of covariance parameters $\btheta_\mrs = (\sigma^2_\mrs,\rho_\mrs, \delta^2_\mrs)\upp$, we obtained $\cL(\hat{\btheta}_\mrs|\by)=$ 397.25.  This shows that the unstandardized weights matrix $\bW_\mun$ provides a substantially better likelihood optimization than $\bW_\mrs$. The marginal variances of the row-standardized model are shown in Figure~\ref{Fig-MargVar}b, and the marginal correlations are shown in Figure~\ref{Fig-MargCorr}b.  The difference between $\cL(\hat{\btheta}_\mun|\by)$ and $\cL(\hat{\btheta}_\mrs|\by)$ indicates that the weights matrix $\bC$ has a substantial effect for these data.  

We optimized the likelihood with a geostatistical model next using a spherical autocorrelation model.  Denote the geostatistical correlation matrix as $\bS$, where
\[
				s_{i,j} = [1 - 1.5(e_{i,j}/\alpha) + 0.5(e_{i,j}/\alpha)^3] \cI(d_{i,j} < \alpha),
\]
and $\cI(\bigcdot)$ is the indicator function, equal to one if its argument is true, otherwise it is zero, and $e_{i,j}$ is Euclidean distance between the centroids of the $i$th and $j$th polygons in Figure~\ref{Fig-ColumbusCrimeMap}.  We included a nugget effect, so $\bSigma_\msp = \sigma^2_\msp\bS + \delta_\msp^2\bI$. For the set of covariance parameters $\btheta_\msp = (\sigma^2_\msp,\alpha, \delta^2_\msp)\upp$, we obtained $\cL(\hat{\btheta}_\msp|\by)=$ 374.61.  The geostatistical model provides a substantially better optimized likelihood than either the unstandardized or row-standardized CAR model. The marginal variances of geostatistical models are stationary (Figure~\ref{Fig-MargVar}c).  The estimated range parameter, $\hat{\alpha}$, is shown by the lower bar in Figure~\ref{Fig-ColumbusCrimeMap}.  Any locations separated by a distance greater than that shown by the bar will have zero correlation (Figure~\ref{Fig-MargCorr}c).

It appears that the geostatistical model provides a much better optimized likelihood than the two most commonly-used CAR models.  Is it possible to find a CAR model to compete with the geostatistical model?  Using Theorem~\ref{CARtheorem}, we created $\bC_\mcg$ and $\bM_\mcg$ as in (\ref{eq:CARproof}) from the positive definite covariance matrix from the geostatistical model, $\bSigma_\msp = (\bI - \rho_\mcg\bC_\mcg)\upi\bM_\mcg$. Here, we have a CAR representation that is equivalent to the spherical geostatistical model.  Letting $\bW_\mcg = \bC_\mcg$, and using $\bSigma_\mcg = \sigma^2_\mcg(\bI - \rho_\mcg\bW_\mcg)\upi\bM_\mcg + \delta^2_\mcg\bI$, we optimized for  $\btheta_\mcg = (\sigma^2_\mcg,\rho_\mcg, \delta^2_\mcg)\upp$.  For $\bSigma_\mcg$ to be positive definite, $\sigma^2_\mcg > 0$,  -1.104 $< \rho_\mcg <$ 1.013, and $\delta^2_\mcg \ge 0$.  Because $\btheta_\mcg = (1,1,0)\upp$ is in the parameter space, we can do no worse than the spherical geostatistical model.  In fact, upon optimizing, we obtained $\cL(\hat{\btheta}_\mcg|\by)=$ 373.95, where $\hat{\sigma}^2_\mcg =$ 0.941, $\hat{\rho}_\mcg =$ 1.01, and $\hat{\delta}^2_\mcg =$ 0, a slightly better optimization than the spherical geostatistical model.  The marginal variances for this geostatistical-assisted CAR model are shown in Figure~\ref{Fig-MargVar}d, and the marginal correlations are shown in Figure~\ref{Fig-MargCorr}d.  Note the rather large changes from Figure~\ref{Fig-MargVar}c to Figure~\ref{Fig-MargVar}d, and from Figure~\ref{Fig-MargCorr}c to Figure~\ref{Fig-MargCorr}d, with seemingly minor changes in $\hat{\sigma}^2_\mcg$, from 1 to 0.941, and in $\rho_\mcg$, from 1 to 1.01. Others have documented rapid changes in CAR model behavior near the parameter boundaries, especially for $\rho_\mcg$ \citep{Besa:Koop:cond:1995, Wall:clos:2004}.


\section{Discussion and Conclusions} \label{sec:discon}

\citet[p. 89]{Hain:spat:1990} provided the most comprehensive comparison of the mathematical relationships between CAR and SAR models.  He provided several results that we restate using notation from Sections \ref{sec:SAR} and \ref{sec:CAR}, and show that some are incorrect or incomplete.

In an attempt to create a CAR covariance matrix from a SAR covariance matrix, assume that $\bB$ is a SAR covariance matrix satisfying conditions \ref{S1}-\ref{S3} and $\bOmega = \bI$ in (\ref{eq:sarcov}). Let $\bM = \bI$ and $\bC$ be symmetric in (\ref{eq:carcov}) [which omits the important case (\ref{eq:bWone})]. Then setting SAR and CAR covariances matrix equal to each other,
\begin{equation} \label{eq:CARequalSAR}
	(\bI-\bC)\upi =[(\bI - \bB)(\bI - \bB\upp)]\upi = (\bI - \bB - \bB\upp + \bB\bB\upp)\upi,
\end{equation}
and \citet{Hain:spat:1990} claims that $\bC$ can be obtained from $\bB$ by setting
\begin{equation} \label{eq:CARfromSAR}
	\bC = \bB + \bB\upp - \bB\bB\upp,
\end{equation}
which is repeated in texts by \citet[][p. 372]{Wall:Gotw:appl:2004} and \citet[][p. 339]{Scha:Gotw:stat:2005}, and in the literature \citep[e.g.,][]{Dorm:etal:meth:2007}.  However, aside from the lack of generality due to assumptions $\bM = \bI$, $\bOmega = \bI$, and symmetric $\bC$, we note that (\ref{eq:CARfromSAR}) is incomplete and too limited to be useful, as given in the following remark.
\begin{remark}
\begin{singlespace}
Condition \ref{C3} in Section~\ref{sec:CAR} is not satisfied for $\bC$ in (\ref{eq:CARfromSAR}) except when $\bB$ contains all zeros. 
\end{singlespace}
\end{remark}
\begin{proof}
\begin{singlespace}
Because $\bB$ has zeros on the diagonal, $\bB + \bB\upp$ will have zeros on the diagonal.  Denote $\bb_i$ as the $i$th row of $\bB$.  Then the $i$th diagonal element of $\bB\bB\upp$ will be the dot product $\bb_i \bigcdot \bb_i$, which will be zero only if all elements of $\bb_i$ are zero. Hence, $\bB + \bB\upp - \bB\bB\upp$ will have zeros on the diagonal only if $\bB$ contains all zeros. \qedhere
\end{singlespace}
\end{proof}

In an attempt to create a SAR covariance matrix from a CAR covariance matrix, assume the same conditions as for (\ref{eq:CARequalSAR}), and that $\bC$ is a CAR covariance matrix satisfying conditions \ref{C1}-\ref{C4}. Let $(\bI - \bC) = \bS\bS\upp$, where $\bS$ is a Cholesky decomposition. \citet{Hain:spat:1990} suggested $\bS = \bI - \bB$ and setting $\bB$ equal to $\bI - \bS$.  However, this is incomplete because condition \ref{S3} in Section~\ref{sec:SAR} will be satisfied only if $\bS$ has all ones on the diagonal, which also has limited use.    

For another approach to relate SAR and CAR covariance matrices, \citet{Hain:spat:1990} described the model $\bF(\bZ-\bmu) = \bH\bvarepsilon$, where $\textrm{var}(\bvarepsilon) = \bV$.  Then $\textrm{E}((\bZ-\bmu)(\bZ-\bmu)\upp) = \bF\upi\bH\bV\bH\upp(\bF\upi)\upp$. Now let $\bF = (\bI -\bC)$, $\bH = \bI$, and $\bV = (\bI - \bC)$ (this appears to originate in \citet{Mart:some:1987}).  The constructed model is really a SAR model except that it violates condition \ref{S2} by allowing $\bV = (\bI - \bC)$.  Alternatively, this can be seen as an attempt to create a SAR model from a CAR model by assuming an inverse CAR covariance matrix for the error structure of the SAR model, which gains nothing.  Because these arguments are unconvincing, and other authors argue that one cannot go uniquely from a CAR to a SAR \citep[e.g.,][]{Mard:maxi:1990}, we can find no further citations for the arguments of \citet{Hain:spat:1990} on obtaining a CAR covariance matrix from a SAR covariance matrix. 

\citet[][p. 409-410]{Cres:stat:1993} provided a demonstration of how a SAR covariance matrix with first-order neighbors in $\bB$ leads to a CAR covariance matrix with third-order neighbors in $\bC$, and claims that, generally, there will be no equivalent SAR covariance matrices for first and second-order CAR covariance matrices.  However, our demonstration in Figure~\ref{Fig-graphModel}c shows that a sparse $\bB$ may be obtained from a sparse CAR model, although it is asymmetric and may not have the usual neighborhood interpretation.  

From Section~\ref{sec:weightsMat}, we showed that pre-specified weights $\bW$ are often scaled by $\rho$, and that $\rho$ is often constrained by the eigenvalues of $\bW$.  However, we have also discussed in Section~\ref{sec:implications} and Figure~\ref{Fig-graphModel}c, that weights can be chosen so that all eigenvalues are zero, for either CAR or SAR models. We have little information or guidance for developing models where all eigenvalues $\bW$ are zero, and this provides an interesting topic for future research. 

\citet{Wall:clos:2004} provided a detailed comparison on properties of marginal correlation for various values of $\rho$ when $\bB$ or $\bC$ are parameterized as $\rho_s\bW$ and $\rho_c\bW$, respectively, but did not develop mathematical relationships between CAR and SAR models. \citet{Lind:Rue:Lind:expl:2011} showed that approximations to point-referenced geostatistical models based on a finite element basis expansion can be expressed as CAR models.  In his discussion of the same, \citet{Kent:disc:2011} noted that, for a given geostatistical model of the Matern class, one could construct either a CAR or SAR model that would approximate the Matern model.  This indicates a correspondence between CAR and SAR models when used as approximations to continuous-space processes, but does not address the relationship between CAR and SAR models on a native areal support.

Our literature review and discussion showed that there have been scattered efforts to establish mathematical relationships between CAR and SAR models, and some of the reported relationships are incomplete on the conditions for those relationships. With Theorems~\ref{SARtheorem} and \ref{CARtheorem} and Corollary~\ref{SAReqCAR}, we demonstrated that any zero-mean Gaussian distribution on a finite set of points, $\bZ \sim \textrm{N}(\bzero,\bSigma)$, with positive-definite covariance matrix $\bSigma$, can be written as either a CAR or a SAR model, with the important difference that a CAR model is uniquely determined from $\bSigma$ but a SAR model is not so uniquely determined. This equivalence between CAR and SAR models can also have practical applications.  In addition to our examples, the full conditional form of the CAR model allows for easy and efficient Gibbs sampling \citep[][p. 163]{Bane:Carl:Gelf:hier:2004} and fully conditional random effects \citep[][p. 86]{Bane:Carl:Gelf:hier:2004}. However, spatial econometric models often employ SAR models \citep{LeSa:Pace:intr:2009}, so easy conversion from SAR to CAR models may offer computational advantages in hierarchical models and provide insight on the role of fully conditional random effects.  We expect future research will extend our findings on relationships between CAR and SAR models and explore novel applications.


\section*{Acknowledgments}

This research began from a working group on network models at the Statistics and Applied Mathematical Sciences (SAMSI) 2014-15 Program on Mathematical and Statistical Ecology. The project received financial support from the National Marine Fisheries Service, NOAA. The findings and conclusions of the NOAA author(s) in the paper are those of the NOAA author(s) and do not necessarily represent the views of the reviewers nor the National Marine Fisheries Service, NOAA. Any use of trade, product, or firm names does not imply an endorsement by the U.S. Government.

\clearpage
\setcounter{equation}{0}
\renewcommand{\theequation}{A.\arabic{equation}}
\setcounter{figure}{0}
\section*{APPENDIX: Propositions on Weights Matrices} \label{app:WMat}

The following proposition is used to show condition~\ref{C1} for CAR models.

\begin{proposition} \label{pro:C1proof}
\begin{singlespace}
  Let $\bSigma = \bA\bM$, where $\bSigma$, $\bA$, and $\bM$ are square matrices, $\bSigma$ is symmetric, $\bA\upi$ exists, and $\bM$ is symmetric and positive definite.  Then $\bSigma$ is positive definite if and only if all of the eigenvalues of $\bA$ are positive real numbers. 
\end{singlespace}
\end{proposition}
\begin{proof}
\begin{singlespace}
($\Longleftarrow$): Let $\bM^{-1/2}$ be the matrix such that $\bM^{-1/2}\bM^{-1/2} = \bM^{-1}$, and let $\bM^{1/2}$ be the matrix such that $\bM^{1/2}\bM^{1/2} = \bM$.  Now, $\bA = \bSigma\bM\upi = \bM^{1/2}[\bM^{-1/2}\bSigma\bM^{-1/2}]\bM^{-1/2}$. Then $\bA$ has the same eigenvalues as $[\bM^{-1/2}\bSigma\bM^{-1/2}]$ because they are similar matrices \citep[][p. 525]{Harv:matr:1997}. If $\bSigma$ is positive definite, then $[\bM^{-1/2}\bSigma\bM^{-1/2}]$ is positive definite, so all eigenvalues of $\bA$ are positive real numbers.

($\Longrightarrow$): Let $\bA = \bU\bLambda\bU\upi$, where the columns of $\bU$ contain orthonormal eigenvectors and $\bLambda$ is a diagonal matrix of eigenvalues that are all positive and real. Because of symmetry, $\bA\bM = \bM\bA\upp = \bM(\bU\upp)\upi\bLambda\bU\upp$, so $\bA\bM(\bU\upp)\upi = \bM(\bU\upp)\upi\bLambda$.  This shows that both $\bU$ and $\bM(\bU\upp)\upi$ have columns that contain the eigenvectors for $\bA$, so each column in $\bU$ has a corresponding column in $\bM(\bU\upp)\upi$ that is a scalar multiple.  Let $\bGamma$ be a diagonal matrix of those scalar multiples, so that $\bM(\bU\upp)\upi = \bU\bGamma$.  Hence, $\bU\upi\bM(\bU\upp)\upi = \bGamma$, and notice that all diagonal elements of $\bGamma$ will be positive because $\bM$ is positive definite. Also $\bU\upi\bM = \bGamma\bU\upp$, so $\bSigma = \bA\bM = \bU\bLambda\bU\upi\bM = \bU\bLambda\bGamma\bU\upp$.  Because $\bLambda$ and $\bGamma$ are diagonal, each with all positive real values, $\bSigma$ is positive definite.
 \qedhere
\end{singlespace}
\end{proof}

Condition \ref{C1} is satisfied by letting $\bSigma_\textrm{CAR}$ in (\ref{eq:carcov}) be $\bSigma$ in Proposition~\ref{pro:C1proof}, by letting $(\bI - \bC)\upi$ in (\ref{pro:C1proof}) be $\bA$ in Proposition~\ref{pro:C1proof} (note that if $(\bI - \bC)\upi$ has all positive eigenvalues, so too does $\bI - \bC$), and by letting $\bM$ in (\ref{eq:carcov}) be $\bM$ in Proposition~\ref{pro:C1proof}.

Next, we show the conditions on $\rho$ that ensure that $(\bI - \rho\bW)$ has either nonzero eigenvalues, or positive eigenvalues.

\begin{proposition} \label{pro:PDrhobounds}
\begin{singlespace}
				Consider the square matrix $(\bI - \rho\bW)$, where $w_{i,i} = 0$. Let $\{\lambda_i\}$ be the set of eigenvalues of $\bW$, and suppose all eigenvalues are real. Then
\begin{enumerate}
\item [(i)] if $\rho \notin \{\lambda_i\upi\}$ for all nonzero $\lambda_i$, then $(\bI - \rho\bW)$ is nonsingular, and
\item[(ii)] assume at least two eigenvalues of $\bW$ are not zero, and let $\lambda_{[1]}$ be the smallest eigenvalue of $\bW$, and $\lambda_{[N]}$ be the largest eigenvalue of $\bW$. If $1/\lambda_{[1]} < \rho < 1/\lambda_{[N]}$, then $(\bI - \rho\bW)$ has only positive eigenvalues. 
\end{enumerate}
\end{singlespace}
\end{proposition}

\begin{proof}
\begin{singlespace}
			Let $\omega$ be an eigenvalue of $(\bI - \rho\bW)$, so the following holds,
\begin{equation} \label{eq:eigen}
 (\bI - \rho\bW)\bx = \omega\bx.
\end{equation}
Let $\bv_i$ be the eigenvector corresponding to $\lambda_i$.  Solving for $\omega_i$ in (\ref{eq:eigen}), let $\bx = \bv_i$. Then,
\[
	\begin{array}{l}
	\bv_i - \rho\bW\bv_i = \omega_i\bv_i,\\
	 \implies \bv_i - \rho\lambda_i\bv_i = \omega_i\bv_i, \\
	 \implies (1 - \rho\lambda_i)\bv_i = \omega_i\bv_i, \\
	 \implies (1 - \rho\lambda_i) = \omega_i. \\
	\end{array}
\]
 Then,
\begin{enumerate}
\item [(i)] $(\bI - \rho\bW)$ is nonsingular if all $\omega_i \ne 0$, so if $\lambda_i = 0$, then $\omega_i = 1$ for all $\rho$, otherwise $(1 - \rho\lambda_i) \ne 0 \implies \rho \ne 1/\lambda_i$ for all nonzero $\lambda_i$.
\item[(ii)] For all $\omega_i > 0$, $(1 - \rho\lambda_i) > 0 \implies 1 > \rho\lambda_i $ for all $i$. If $\lambda_i < 0$, then $1/\lambda_i < \rho$, and if $\lambda_i > 0$, then $\rho < 1/\lambda_i$. For all negative $\lambda_i$, only $1/\lambda_{[1]} < \rho$ will ensure all $(1 - \rho\lambda_i) > 0$, and for positive $\lambda_i$, only $\rho < 1/\lambda_{[N]}$ will ensure all $(1 - \rho\lambda_i) > 0$.  Hence $1/\lambda_{[1]} < \rho < 1/\lambda_{[N]}$ will guarantee that all eigenvalues of $(\bI - \rho\bW)$ are positive. \qedhere
\end{enumerate}
\end{singlespace}
\end{proof}


\end{document}